\documentclass[12pt]{article}
\usepackage{authblk}%%%for authors
\usepackage{mathrsfs}
\usepackage{amsthm}
\usepackage{amssymb}
\usepackage{latexsym}
\usepackage{amsmath,amsfonts}
\usepackage{mathrsfs}
\usepackage{cases}
\usepackage{latexsym,bm}
\usepackage{indentfirst}
\usepackage{xcolor}
\usepackage{ifpdf}
\usepackage{graphicx}
\usepackage{epstopdf}
\usepackage{epsfig}
\usepackage{psfrag}
\usepackage{enumitem}
\usepackage{epstopdf}
\usepackage{verbatim}
\usepackage{color}

\usepackage[
pdfauthor={Lu},
pdftitle={Tight cuts in matching covered graphs},
pdfstartview=XYZ,
bookmarks=true,
colorlinks=true,
linkcolor=blue,
urlcolor=blue,
citecolor=blue,
bookmarks=true,
linktocpage=true,
hyperindex=true
]{hyperref}

\newtheorem{lem}{Lemma}[section]
\newtheorem{thm}[lem]{Theorem}
\newtheorem{cor}[lem]{Corollary}

\newtheorem{pro}[lem]{Proposition}

\newcounter{countclaim}
\def\inclaim{\addtocounter{countclaim}{1}
	{\vspace{0.2 cm}\noindent {\bf Claim \thecountclaim}: }}

\usepackage{color}
\begin{document}

\title{Tight cuts in matching covered graphs\footnote{ The research
		is partially supported by  NSFC (Nos. 12271235 and 12371340).
} }

\author[1]{\small Fuliang Lu\thanks{Corresponding author.
Email: flianglu@163.com.}}

\author[2]{\small Fengming Dong\thanks{Email: fengming.dong@nie.edu.sg and donggraph@163.com.}}

\affil[1]{\footnotesize School of Mathematics and Statistics,
Minnan Normal University, Zhangzhou, China}

\affil[2]{\footnotesize National Institute of Education,
	Nanyang Technological University, Singapore}

\date{}

\maketitle

\begin{abstract}
    An edge cut $C$ of a graph $G$ is {\em tight} if $|C\cap M|=1$ for every perfect matching $M$ of $G$. Barrier-cuts and 2-separation cuts, also referred to as {\em ELP-cuts}, are two important types of tight cuts in matching covered graphs.
     Edmonds, Lov\'asz and Pulleyblank [Brick decompositions and the matching rank of graphs, Combinatorica 2(3) (1982) 247-274] proved that if a matching covered graph has a non-trivial tight cut, then it also has a non-trivial ELP-cut.
    In confirmation of a conjecture proposed by
    Carvalho, Lucchesi and Murty, Chen et. al. [Laminar tight cuts in matching covered graphs, J. Comb. Theory, Ser. B, 150 (2021) 177-194]
    showed that if $C$ is a non-trivial tight cut of a matching covered graph $G$, then $G$ has at least one $C$-sheltered non-trivial barrier or a 2-separation cut that is laminar with $C$.

    In this paper,  we present a complete  characterization of   non-trivial tight cuts in matching covered graphs,
    from which the result of Chen et. al. can be derived directly. Moreover, we show that the lower bound of the number of  $C$-sheltered non-trivial barrier or a 2-separation cut that is laminar with $C$ in the result of Chen et. al.
    is sharp.

\end{abstract}


\vskip 0.2in \baselineskip 0.1in
%%%%%
%%%%%
%%%%%%%%%%%%%%%%%%%%%%%%%%%%%%%%%%%%%%%%%%%%%%%%%%%%%%%%%%%%%%%%%%%%%%%%%%%%%%%%%%%%%%%%%%%%%%%%%%%%%%%%%%%%%%%%%%%%%%%%%%%%%%%lufuliang

\section{Introduction}

Graphs considered in this paper may have multiple edges, but no loops. We follow \cite{BM08} for undefined notations and terminologies.

For any graph $G$, let $V(G)$ and
$E(G)$ be its vertex set and edge set, respectively.
%Let $G$ be a graph with the vertex set $V(G)$ and the edge set $E(G)$.
  For $X,Y\subseteq V(G)$, by $E[X,Y]$ we mean the set of edges in $G$ each of which has one end point in
  $X$ and the other end point in $Y$.
  % with one end vertex in $X$ and the other end vertex in $Y$.
  Clearly,  $\partial(X):=E[X,\overline{X}]$ is an edge cut of $G$
for any nonempty proper subset $X$ of $V(G)$,
  where $\overline{X}=V(G)\backslash X$.
  The vertex sets $X$ and $\overline{X}$ are shores of $\partial(X)$.
  An edge cut $\partial(X)$ is {\em trivial} if either $|X|=1$ or $|\overline{X}|=1$.
  For any nonempty subset $X$ of $V(G)$,
  the subgraph of $G$ induced by $X$ is denoted by $G[X]$; the set of all vertices of $V(G)\setminus X$ that are adjacent to some vertex in $X$ is denoted by $N_G(X)$.
  Two edge cuts $\partial(X)$ and $\partial(Y)$ {\em cross}
  if each of $X\cap Y$, $X\cap \overline{Y}$, $\overline{X}\cap Y$ and $\overline{X}\cap \overline{Y}$ is not empty. Two edge cuts are {\em laminar} if they do not cross.

Let $\partial(X)$ be an edge cut of $G$. Denoted by $G/(X\rightarrow x)$ the graph obtained from $G$ by contracting $X$ to a singleton $x$ (and removing any resulting loops). The graphs $G/(X\rightarrow x)$ and $G/(\overline{X}\rightarrow \overline{x})$ are $\partial(X)$-contractions of $G$.

A connected non-trivial graph $G$ is {\em matching covered} if each edge in $G$ lies in a perfect matching. An edge cut $\partial(X)$ is {\em tight} of $G$ if $|\partial(X)\cap M|=1$ for every perfect matching $M$ of $G$. A trivial edge cut is obviously a tight cut. A matching covered graph without non-trivial tight cuts is called
a {\em brace} if it is bipartite, and
is called a {\em brick} if it is non-bipartite.
There is a procedure called a {\em tight cut decomposition}, due to Lov\'{a}sz \cite{Lovasz87}, which can be applied to a matching covered graph $G$ to produce a list of unique braces and bricks. In particular, any two applications of the tight cut decomposition procedure yield the same number of bricks \cite{Lovasz87}, which is called the brick number of $G$.
Using the tool of tight cut decomposition, many  problems about matching covered graphs can be reduced to braces and bricks.

Let $G$ be a matching covered graph.
We denote by
$o(H)$ the number of odd components of a graph $H$.
A nonempty vertex set $B$ of $G$
%that has a perfect matching
is a {\em barrier} if $o(G-B)=|B|$.
An edge cut $C$ of $G$ is a {\em barrier-cut} if there exists a barrier $B$ of $G$ and an odd component $Q$ of $G-B$ such that $C=\partial(V(Q))$.
A vertex set $S$ of $G$ is a {\em 2-separation} if $|S|=2$, $G-S$ is disconnected and each of the components of $G-S$ is even.
Let $\{u,v\}$ be a 2-separation of $G$, and let us divide the components of $G-\{u,v\}$ into two nonempty subgraphs $G_1$ and $G_2$. The cuts $\partial(V(G_1)+u)$ and $\partial(V(G_1)+v)$ are both {\em 2-separation cuts} associated with $\{u,v\}$ of $G$.
Barrier-cuts and 2-separation cuts, called {\em ELP-cuts}, play an important role
in
tight cut decomposition (it can be checked that  ELP-cuts are tight cuts).
The following theorem,
called ELP-Theorem, due to Edmonds, Lov\'{a}sz, and Pulleyblank \cite{ELP82}, is a fundamental result of matching theory.

\begin{thm}[\cite{ELP82}]\label{thm:ELP-th}
    Every matching covered graph that has a non-trivial tight cut has a non-trivial barrier or a 2-separation.
\end{thm}

The proof of Theorem \ref{thm:ELP-th} given by Edmonds, Lov\'{a}sz, and Pulleyblank is based on linear programming techniques \cite{ELP82}. Szigeti \cite{Szi02}
provided a purely graph-theoretical
proof of this theorem.
Carvalho, Lucchesi and Murty \cite{CLM18} gave an alternative proof of ELP-Theorem by using properties of barriers in graphs that has a perfect matching.

\def \ELP {{\cal ELP}}

  A barrier $B$ is $C$-{\em sheltered} if $B$ is a subset of a shore of $C$.
  Note that for a tight cut $C$,
  a $C$-sheltered barrier-cut is
  laminar with $C$.
Carvalho, Lucchesi and Murty \cite{CLM18} conjectured that given any non-trivial tight cut $C$ in a matching covered graph that is not an ELP-cut, there exists a non-trivial ELP-cut $D$  that is laminar with
%does not cross
$C$, and proved this conjecture
for bicritical graphs and also for matching covered graphs
with brick number at most two.
For a matching covered
graph $G$ and a non-trivial tight cut $C$,
let $\ELP_G(C)$
be the collection of
$C$-sheltered non-trivial barrier-cuts
and 2-separation cuts that are
laminar with $C$.
Chen et. al. proved
the conjecture of Carvalho, Lucchesi and Murty~\cite{CLM18},
as stated below.
% by presenting the following theorem.

\begin{thm}[\cite{CFLLZ20}]\label{thm:lam}
Let $C$ be a non-trivial tight cut of a matching covered graph $G$. Then $|\ELP_G(C)|\geq 1$.
\end{thm}

Let $\partial(X)$ be an edge cut of a matching covered  graph $G$,
and let
$ \mathcal{F}$ be
the set of  2-separations $F$ of $G$
such that $|F\cap X|=1$.
We  say $\partial(X)$ is a {\em generalized 2-separation-cut} ($GS$-cut for short)  if
 %two 2-separations $F$ and $F'$ in $ \mathcal{F}$,
the following
 conditions are satisfied
for any $F, F'\in \mathcal{F}$:

 \vspace{-3 mm}

 \begin{enumerate}
 	\item [1)]
 %for any two 2-separations $F$ and $F'$ in $ \mathcal{F}$,
 there exists a set of  2-separations $\{F_1,F_2,\dots, F_s\}
 \subset \mathcal{F}$
such that
%$ \mathcal{F'}\subset \mathcal{F}$,
$F=F_1$, $F'=F_s$  and $|F_{i}\cap F_{i+1}|=1$ for $i=1,2,\dots,s-1$;
and

\item [2)] if $|F\cap F'|=1$
and $Y$ (resp. $Y'$)  is the vertex set of a component of $G-F$ (resp. $G-F'$) such that $Y\subset Y'$,   then all vertices in each odd component of $G[Y'\setminus (Y\cup F)]$ lie in the shore of $\partial (X)$ that is different from the one containing the only vertex in
$F\cap F'$, and
all vertices in each even component of $G[Y'\setminus (Y\cup F)]$ lie in  the same shore of $\partial (X)$;
if $Y\cup F$ contains no 2-separations in $\mathcal{F}\setminus \{F\}$, then $Y\subset X$ or  $Y\subset \overline{X}$.
\end{enumerate}
Specially, a 2-separation cut is a  $GS$-cut. We say $ \mathcal{F}$ is
the set of  2-separations  of $G$ associated with $\partial (X)$. It should be noted that if $\partial(X)$ is a $GS$-cut of $G$ and $e\in \partial(X)$, then   at least one end point of $e$ lies in a  2-separation $F$ of   $G$ satisfying $|F\cap X|= 1$.
For any $F\in {\cal F}$,
if there exists a vertex $w\in V(G)$
such that $F\cap F'\subseteq \{w\}$
for each $F'\in {\cal F}\setminus \{F\}$,
then $F$ is called an
{\it end-2-separation} of ${\cal F}$. Obviously, if $F$ is an end-2-separation of ${\cal F}$, then there exists a
component of $G-F$, say $G_1$, such that $V(G_1)\cup F$
 contains no 2-separations in ${\cal F}\setminus \{F\}$.

\begin{figure}[h]
	\centering
	\includegraphics[width=6cm] {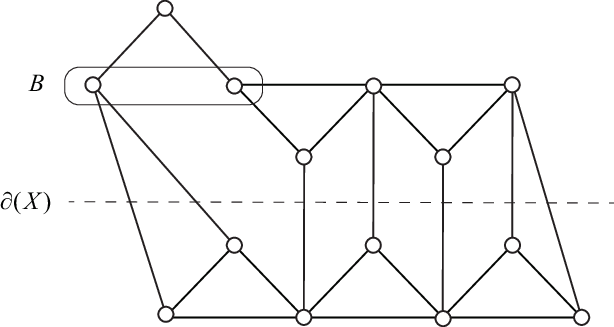}
	\caption{an essential $GS$-cut}
	\label{f1}
\end{figure}

Let $C=\partial(X)$ be a non-trivial tight cut of $G$ and $B$ be a $C$-sheltered non-trivial barrier such that $B\subset \overline{X}$. By Corollary~\ref{cor:connected}, $G[{X}]$ is connected. Then $X\subset V(G_1)$
for some  component $G_1$
of $G-B$.
The contraction of the barrier $B$ means contracting $\overline{ V(G_1)}$ into a single vertex. In fact, the contraction of $B$ results in a $\partial{ (V(G_1))}$-contraction of $G$.

\iffalse
{\color{red}
An edge cut $C$ of $G$ is an {\em essential $GS$-cut} if the edge cut $C$ is a $GS$-cut of the graph gotten by the  contractions of some $C$-sheltered non-trivial barriers (see Figure \ref{f1}).
Specially,  a  $GS$-cut  is also an essential $GS$-cut. It should be note that if an  essential $GS$-cut $\partial (X)$ that is not a  $GS$-cut, then, by
contracting some $\partial (X)$-sheltered barriers, we get a graph $G'$, and $\partial (X\cap V(G'))$ is a $GS$-cut of $G'$. Denote by $S$ the set of vertices in $G'$ that gotten by the contraction of the $\partial (X)$-sheltered barriers, that is $S=V(G')\setminus V(G)$. Then every vertices in $S$ lie in some 2-separation in $ \mathcal{F}$, where $ \mathcal{F}$ is the set of 2-separation that $\partial (X)$ associated with.}
\fi

 Let $\partial(X)$ be an edge cut of $G$. Assume that ${\cal B}=\{B_1,B_2,\ldots,B_s\}$ is a set of $\partial(X)$-sheltered non-trivial barriers of $G$ such that $B_j\cap B_k=\emptyset$
for each pair of $j,k\in \{1,2\ldots,s\}$
	with $j\ne k$,
% for $j\ne k$ and $\{j,k\}\subset \{1,2\ldots,s\}$,
$B_i\subset X$ for all $i=1,2\ldots,t$ and $B_i\subset \overline{X}$ for  all $i=t+1,t+2\ldots,s$.
Let $G'$ denote the resulting graph
after contracting $B_i$  into a vertex $b_i$
for all $i=1,2,\cdots,s$,
and let
$X'=(X \setminus \cup_{i=1}^t B_i)
\cup \{b_i: 1\le i\le t\}$.
%\cup (\cup_{i=1}^t \{b_i\})\setminus (\cup_{i=1}^t B_i)$.
We say
the edge cut $\partial(X)$  is an {\em essential $GS$-cut} of $G$ if
$\partial(X')$ is a  $GS$-cut of $G'$ such that every vertex $b_i$ lie in some 2-separation $F_i$ of $G'$, where   $|F_i
\cap X'|=1$ (see Figure \ref{f1}). We say ${\cal B}$ is the set of barriers associated with $\partial(X)$.
Specially, if ${\cal B}=\emptyset$, then  an  essential $GS$-cut is a $GS$-cut.

Carvalho, Lucchesi and Murty \cite{CLM17} showed that in a matching covered graph with
brick number at most $2$,
every non-trivial tight
cut is either a barrier-cut or  an essential $GS$-cut.
 In this paper, we show
 that the result holds for all  matching covered graphs.

\begin{thm}\label{main}In a matching covered graph, every non-trivial tight cut is either a barrier-cut or an essential $GS$-cut.
\end{thm}

%an edge cut $C$ of a bicritical graph $G$ is a non-trivial tight cut but not 2-separation cut if and only if $C$ is an NS-cut (defined in Section 3).

\section{Preliminaries}

%A vertex set $S$ of a connected graph $G$ is called a $k$-vertex cut if $|S|=k$ and $G-S$ is disconnected.
In this section, we will present several useful results.
The following theorem
was established by Tutte.

\begin{thm}[\cite{Tutte47}]\label{thm:Tutte} A graph $G$ has a perfect matching if and only if $o(G-S)\le|S|$ for every $S\subseteq V(G)$.
\end{thm}

 The following  proposition
 follows Theorem \ref{thm:Tutte} directly.

\begin{pro}\label{pro:B_in_mc}
     Let $G$ be a matching covered graph and let $B$ be a non-trivial barrier of $G$. Then $E(G[B])=\emptyset$ and $G-B$ contains no even components.
\end{pro}

The following are some properties of tight cuts in matching covered graphs that will be used later.

\begin{pro}[\cite{ELP82}]
	\label{pro:ctight}
    Let $G$  be a matching covered graph and let $\partial(X)$  and   $\partial(Y)$ be two tight cuts such that $|X\cap Y| $ is odd.
    Then $\partial(X\cap Y)$ and $\partial(X\cup Y)$ are also tight cuts of $G$. Furthermore, no edge connects $X\cap \overline{Y}$ to $\overline{X}\cap Y$.
\end{pro}

\begin{pro}[\cite{Lovasz87}]
	\label{pro:also_tight}
    Let $G$ be a matching covered graph, and let $C$ be a tight cut of $G$. Then, both $C$-contractions are matching covered. Moreover, if $G'$ is a $C$-contraction of $G$, then a tight cut of $G'$ is also a tight cut of $G$. Conversely, if a tight cut of $G$ is an edge cut of $G'$, then it is also tight in $G'$.
\end{pro}

%For $X\subset V(G)$, denote by $G[X]$ the subgraph induced by $X$ and let $N(X)=\{y\in \overline{X}: xy\in E(G)~\mbox{and}~ x\in X\}$.

It is known that every matching covered graph with four or more vertices is 2-connected \cite{Lovasz86}. The following corollary can be derived by Proposition \ref{pro:also_tight} directly (see Corollary 1.6 in \cite{CFLLZ20} for example).

\begin{cor}
	\label{cor:connected}
    Let $G$ be a matching covered graph, and let $C=\partial(X)$ be a tight cut of $G$. Then, both $G[X]$ and $G[\overline{X}]$ are connected.
\end{cor}

\begin{lem}[\cite{CFLLZ20}]\label{lem:2-sep}
    Let $C:=\partial(X)$ be a 2-separation cut of a matching covered graph $G$, associated with a 2-separation $\{u_1,u_2\}$, where $u_1\in X$ and $u_2\in\overline{X}$,
    and let $S_H$ denote either a 2-separation or a barrier of the $C$-contraction $H:
    =G/(\overline{X}\rightarrow \overline{x})$
        of $G$    and let
    \begin{equation*}
    	S:=
        \begin{cases}
            S_H,&~\mbox{if}~\overline{x}\notin S_H; \\
            (S_H-\overline{x})+u_2,
            \quad &~\mbox{if}~\overline{x}\in S_H.
        \end{cases}
    \end{equation*}
    If $S_H$ is a barrier
    (resp. 2-separation)
    of $H$, then $S$ is a barrier
    (resp. 2-separation) of $G$.
    % and if $S_H$ is a 2-separation of $H$, then $S$ is a 2-separation of $G$.
\end{lem}

\iffalse
\begin{pro}\label{pro:also_tight}
    Let $B$ be a non-trivial barrier of   a matching covered graph $G$, $G_1$ be a non-trivial component of $G-B$ and
    $X\subset V(G_1)$. Let $G'=G/(V(G_1)\rightarrow t)$.
    If $\partial(X)$ is a barrier-cut of $G'$ associated with the barrier $B'$, then $(B\cup B')\setminus\{t\}$ is a barrier of $G$ such that $G[X]$ is a component of $G-(B\cup B_1)\setminus\{t\}$, or $B'$ is also a barrier of $G$ such that $G[\overline{X}]$ is a component of $G-B'$.
\end{pro}
\begin{proof}If $t\not B'$, then $t$ lie in some component of $G-B'$, say $G_1'$. As
$\partial(X)$ is a barrier-cut of $G'$, $V()$

\end{proof}
\fi
\section{Essential $GS$-cuts}

\begin{pro}\label{tight} Let $\{x,y\}$ be a $2$-separation of a matching covered graph $G$ with $xy\in E(G)$.
Assume that
$G_1$ and $ G_2$ are connected induced subgraphs of $G$
such that
$V(G)=V(G_1)\cup V(G_2)$
and $V(G_1)\cap V(G_2)=\{x,y\}$.
For $i=1,2$, let
$X_i$ be a
subset of $V(G_i)$
such that
$|X_i|$ is odd, $x\in X_i$ and $xy\in \partial (X_i)$.
Then, for $X=X_1\cup X_2$,
$\partial (X)$ is tight  in  $G$ if and only if $\partial (X_i)$ is tight in $G_i$
for both $i=1$ and $2$.
%and $\partial (X_2)$ is tight in $G_2$.
\end{pro}

\begin{proof}
As $\{x,y\}$ is a 2-separation of $G$,
both $|V(G_1)|$ and $|V(G_2)|$
are even. Thus, it can be verified that
for any perfect matching $M$ of $G$,
$M\cap \partial(V(G_1)\setminus\{ y\})$ contains
either exactly an edge $xu$ for some $u\in V(G_2)\setminus \{x\}$,
or exactly an edge $yv$
for some $v\in V(G_1)\setminus \{y\}$.
It follows that
 $\partial(V(G_1)\setminus\{ y\})$ is a tight cut of $G$. So $G_1$ and $G_2$ are matching covered graphs  by Proposition \ref{pro:also_tight}.
 Note that $\partial (X_1)\subset \partial (X)$.
 Thus, if $\partial (X)$ is tight in $G$, then
 $\partial (X_1)$ is tight in $G_1$,
 and so is $\partial (X_2)$ in $G_2$.

Now we show the sufficiency.
As $|X_i|$ is odd for $i=1,2$, and $X_1\cap X_2=\{x\}$, $|X|$ is odd. Suppose to the  contrary that $\partial (X)$ is not tight,
that is $|M\cap \partial(X)|\ne 1$
for some perfect matching $M$ of $G$. As $|X|$ is odd,
$|M\cap \partial(X)|$ is odd, implying that  $|M\cap \partial(X)|\geq 3$.
If $xy\in M$, then $M\cap E(G_i)$ is a perfect matching of $G_i$ ($i=1,2$); let $M_i=M\cap E(G_i)$. If $xy\notin M$, assume that $\{xx_1,yy_1\}\subset M$. Then $\{x_1,y_1\}\subset V(G_1)$ or $\{x_1,y_1\}\subset V(G_2)$,  as $\{x,y\}$ is a 2-separation. Adjust notation so that  $\{x_1,y_1\}\subset V(G_1)$. Then $M\cap E(G_1)$ is a perfect matching of $G_1$, and $(M\cap E(G_2))\cup\{xy\}$ is a perfect matching of $G_2$;  let $M_1=M\cap E(G_1)$, $M_2=(M\cap E(G_2))\cup\{xy\}$.
Note that $\partial_G(X)=\partial_{G_1}(X_1)\cup  \partial_{G_2}(X_2)$. So at least one of $|M_1\cap\partial_{G_1}(X_1)|$ and $|M_2\cap\partial_{G_2}(X_2)|$ is great than 1. On the other hand, $\partial (X_1)$  and $\partial (X_2)$ are tight in $G_1$  and  $G_2$, respectively. This is a contradiction.
\end{proof}

\begin{pro}\label{two}
Every essential $GS$-cut
of a matching covered graph
is tight.
\end{pro}

\begin{proof}
	Let $G$ be a  matching covered  graph.
We first prove the following claim on $GS$-cuts.

\inclaim Every $GS$-cut of $G$
is tight in $G$.

Let $\partial(X)$ be
a $GS$-cut of a $G$
	and  $\mathcal{F}$ be the set of  2-separations of $G$ associated with $\partial(X)$.
For any  2-separation
	$F=\{x,y\}\in \mathcal{F}$,
we may assume that $xy\in E(G)$.
It is because if $\partial(X)$ is tight
in $G$ with $xy\in E(G)$,
then $\partial(X)$ is also tight in
$G-xy$ by the definition of a tight cut (it can be checked that  $G$ is matching covered if and only if $G-xy$ is matching covered).

We  denote the number of 2-separation cuts crossing $\partial(X)$ by $\lambda$.
 We will show the claim by induction on $\lambda$.
As
 a 2-separation results in at least two 2-separation cuts,
 we have $\lambda\geq 2$.
 If $ \lambda=2$,  assume that $D\in \mathcal{F}$.
 Then $\partial(X)$ is a 2-separation cut associated with $D$. So $\partial(X)$ is tight.
 We assume that Claim 1 holds whenever
  $\lambda<k$, where $k\ge 3$.
  Now we consider the case $\lambda=k$.

  Let $D_1$ be an end-2-separation in $ \mathcal{F}$ such that one  component of $G-D_1$, say $H_1$,
  satisfies the condition that
  $V(H_1)\cap F=\emptyset$
  for each  $F\in \mathcal{F}$.

  Then by the definition of a $GS$-cut, all the vertices of $V(H_1)$ lie in one shore of  $\partial(X)$.
  Therefore, $V(H_1)\cup \{x_1\}$ is a 2-separation cut of $G$, where $x_1\in D_1$, and $x_1$ and all the vertices of $V(H_1)$ lie in the same shore of  $\partial(X)$. Let $G_1=G/(V(H_1)\cup \{x_1\} ) $ and $G_2=G/(\overline{V(H_1)\cup \{x_1\}}) $. By the definition of a $GS$-cut, $ \partial(X\cap V(G_1) )$ is also a $GS$-cut of $G_1$.
  By inductive hypothesis, the  edge cut $ \partial(X\cap V(G_1) )$ is tight in $G_1$. Note that $ \partial(X\cap V(G_2) )$ is a trivial edge cut of $G_2$, which is tight.
   So Claim 1 follows from Proposition \ref{tight}.

Now we are going to complete the proof. Let $\partial(Y)$ be an essential $GS$-cut.
Then we may contract some $\partial(Y)$-sheltered non-trivial barriers to get a $GS$-cut $\partial(Y')$  of the resulting graph, say $G'$.
It is known that barrier-cuts are tight.  By Claim 1,
    $\partial(Y')$ is tight in $G'$.
    So  $\partial(Y)$ is a tight cut of $G$ by applying Proposition \ref{pro:also_tight} repeatedly.
\end{proof}

Now we are going to establish the main result in this section.

\begin{thm}\label{thm:lamnum}
Let $C$ be a non-trivial $GS$-cut of a matching covered graph $G$. %\red{whose brick number is at least two}.
%at least {\color{red}two bricks}.
Then, $|\ELP_G(C)|\ge 2$.
%$G$ has at least two $C$-sheltered non-trivial barriers or a 2-separation cuts which is laminar with $C$.
\end{thm}

\begin{proof}
Let $C=\partial(X)$ be a
non-trivial $GS$-cut of $G$, where $X\subset V(G)$.
Assume that the brick number of
$G$ is at most one.
Then $C$ is a barrier-cut by Proposition 4.18 in \cite{lm2024}. This is a contradiction to the assumption that $C$ is a  $GS$-cut.

Now we assume that  the brick number of
$G$ is at least two, and   $\mathcal{F}$ is the set of  2-separations  of $G$ associated with $C$.
% each 2-separation.
By the definition of a  $GS$-cut,
there exists an end-2-separation $F_1$ of $\mathcal{F}$.
Then there is a component $G_1$
of $G-F_1$ such that
$V(G_1)\cup F_1$ contains no 2-separation in $ \mathcal{F}\setminus \{F_1\}$.
%where  $G_1$ is a component of $G-F_1$.
Then $V(G_1)$, together with a vertex in $F_1$,
forms one shore of a 2-separation cut
which is associated with $F_1$ and  laminar with $C$,
as all the vertices in $G_1$ lie in the same shore of $C$.

If $\overline{V(G_1)}$ contains no 2-separations in $ \mathcal{F}\setminus \{F_1\}$, then $\overline{V(G_1)\cup F_1}$, together with a vertex in $F_1$,  forms one shore of a 2-separation cut which is associated with $F_1$
and laminar with $C$, as either $\overline{V(G_1)\cup F_1}\subset X$ or $\overline{V(G_1)\cup F_1}\subset \overline{X}$.
So we assume that there exists a
 component $G_2$
of $G-F_1$ such that
 $G_1\neq G_2$,  and $V(G_2)\cup F_1$
contains some 2-separation in $ \mathcal{F}\setminus \{F_1\}$.
%where $G_2$ is a component of $G-F_1$ other than $G_1$.
Let $F_2$ be an end-2-separation of $\mathcal{F}$ with $F_2\subset (V(G_2)\cup F_1)$.
Then, there is a component
$G_3$ of $G-F_2$ such that
$V(G_3)\cup F_2$ contains no 2-separation in $ \mathcal{F}\setminus \{F_2\}$.

Observe that all the vertices of $G_3$ lie in the same shore of $C$ by the definition of a  $GS$-cut.
Therefore, $V(G_3)$, together with a vertex in $F_2$, forms one shore of a 2-separation cut
which is associated with $F_2$ and  laminar with $C$. So the result follows.
\end{proof}

We end this section with the following remarks.

{\bf Remarks}.
(i) The lower bound in Theorem \ref{thm:lamnum} is sharp. A graph $G$ with four or more vertices is {\em bicritical} if for any two distinct vertices $u$ and $v$ in $G$, $G-\{u, v\}$ has a perfect matching. It is known that every barrier in a bicritical graph is trivial (see Theorem 5.2.5 in \cite{Lovasz86} for example).
Let $\{x,y\}$ be a 2-separation of a matching covered graph $G$ with $xy\in E(G)$. Assume that
$G_1$ and $ G_2$ are connected induced subgraphs of $G$
such that
$V(G)=V(G_1)\cup V(G_2)$
and $V(G_1)\cap V(G_2)=\{x,y\}$. We also say $G$ is obtained by the {\em edge splicing} of $G_1$ and $ G_2$.
If $G_1$ and $ G_2$ are bicritical, then it can be checked that $G$ is
also bicritical.

Let $H_n$ be the graph obtained from two vertex-disjoint paths $v_1v_2\ldots v_{2n+1}$ and $u_1u_2\ldots u_{2n+1}$ by adding edges
in the following set:
\begin{eqnarray*}
& &\{v_iv_{i+2},v_iu_{i},v_iu_{i+2},u_iu_{i+2}:i=1,3,\ldots,2n-1 \}\\
&\cup & \{v_iu_{i+1}:i=1,2,\ldots,2n\}
\cup
\{v_{2n+1}u_{2n+1}\},
\end{eqnarray*}
as shown in Figure \ref{f2}. Note that
$\mathcal{F}=\{\{v_i,u_{i+2}\},\{v_{i+2},u_{i+2}\}:i=1,2,\ldots,2n-1\}$  is the set of  2-separations of $H_n$.
It can be checked that
$H_n$ can be obtained by edge splicing of $2n$ $K_4$ in some manner, and thus it is bicritical.
Therefore, every non-trivial ELP-cut of $H_n$
is a 2-separation cut associated with
some 2-separation in $\mathcal{F}$.
Note that $\partial(\{v_1,v_2,\ldots, v_{2n+1}\})$ is a $GS$-cut of $H_n$. Thus, $ \partial(u_1,u_2,u_{3})$ and $ \partial(v_{2n-1},v_{2n+1},v_{2n+1})$ are the only
ELP-cuts laminar with $\partial(\{v_1,v_2,\ldots, v_{2n+1}\})$.

\begin{figure}[h]
	\centering
	\includegraphics[width=6cm] {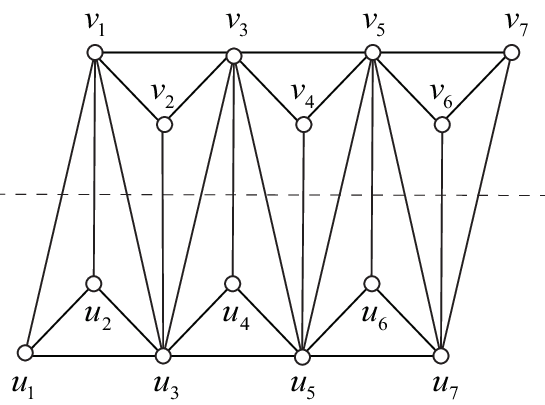}
	\caption{$H_3$}
	\label{f2}
\end{figure}

(ii)
For any essential $GS$-cut
$C=\partial (X)$ that is not a $GS$-cut,
if ${\cal B}$
%=\{B_1,B_2,\ldots,B_s\}$
is the set of  barriers of $G$ associated with $C$, then ${\cal B}\neq \emptyset$. As every barrier in ${\cal B}$ is $C$-sheltered, $|\ELP_G(C)|\geq 2$ when
$|{\cal B}|\geq 2$. If ${\cal B}=\{B_1\}$ and the vertex
produced by the contraction of $B_1$ lies in every 2-separation  $F$ with $|F\cap X|=1$, then $|\ELP_G(C)|$ may
equal to $1$.

Let $H_n' $ be the graph
obtained from two vertex-disjoint  paths $v_1v_2\ldots v_{2n+1}$($n\geq 4$, $n$ is even) and
$u_1u_0u_2$
 by adding edges
in the following set:
\begin{eqnarray*} & &\{v_iv_{i+2}:i=1,3,\ldots,2n-1 \}\cup \{u_1v_{2i}:i=2,4,\ldots,n\}
	\\
&\cup &
\{u_2v_{2i}:i=1,3,\ldots,n-1 \}\cup \{ u_1v_1,u_2v_{2n+1}\},
%u_0u_1,u_0u_2 \}.
\end{eqnarray*}
as shown in Figure
\ref{f4} for the case $n=4$.
In $H_n'$, the vertex set $\{u_1,u_2 \}$ is a barrier. By contracting $\{u_1,u_2 \}$ into a vertex $u$, we get a bicritial graph in which every 2-separation contains $u$. It can be checked that $\partial(\{v_1,v_2,v_3\})$ is an essential $GS$-cut, and $\partial(\{v_1,v_2,\ldots, v_{2n+1}\})$ is the only ELP-cut  laminar with $\partial(\{v_1,v_2,v_3\})$.

\begin{figure}[h]
	\centering
	\includegraphics[width=6cm] {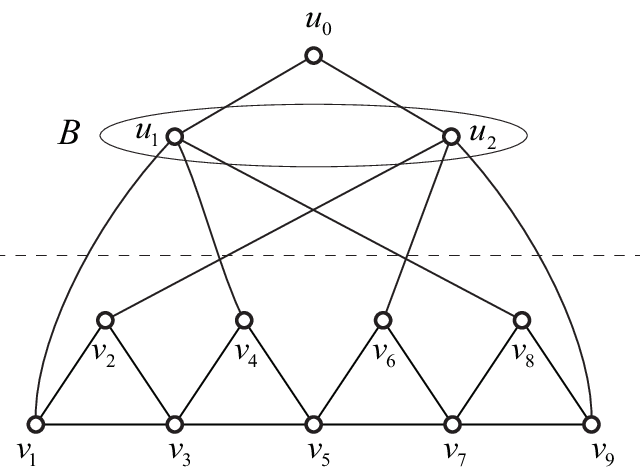}
	\caption{$H_4'$}
	\label{f4}
\end{figure}

\section{Proof of Theorem \ref{main}}

In this section, we will complete the proof of Theorem \ref{main}.

\begin{proof}Let $G$ be a matching covered graph and let
$C=\partial(X)$ be a non-trivial tight cut of $G$.
We will prove the theorem by induction on the number of non-trivial tight cuts in $G$.
If $G$ contains exactly one non-trivial tight cut,
then it is a  non-trivial barrier-cut by
Theorem \ref{thm:ELP-th} (a 2-separation will result in at least two   2-separation cuts). The result follows.
%, $C$ is a non-trivial barrier-cut or a 2-separation cut. The result holds.
Assume that the theorem holds for all the matching covered graphs with at most $k$ non-trivial tight cuts,
 where $k\ge 1$.
 Now we consider the case when $G$ contains exactly
$k+1$ non-trivial tight cuts. Assume that $C$ is not a non-trivial  barrier-cut. We will show that $C$ is  an essential $GS$-cut.

Let us first prove the following claim.

\noindent {\bf Claim  1}: $G$ has a $C$-sheltered non-trivial barrier-cut or a 2-separation cut that is laminar with $C$, that is $\ELP_G(C)\neq \emptyset$.

\begin{proof}Suppose, to the contrary, that for every non-trivial barrier $B$ of $G$, $B\cap X\neq \emptyset$ and $B\cap \overline{X}\neq \emptyset$, and for every 2-separation cut $D$ of $G$, $D$ crosses with $C$.

Firstly, we assume that $G$ is a bicritical graph. Let $F$ be a 2-separation of $G$ such that $F\cup Y$ contains no 2-separations of $G$ other than $F$, where $Y$ is an even component of  $G-F$.

We are now going to prove that
$Y\cap \overline{X}\neq \emptyset$.
Suppose that
$Y\cap \overline{X}= \emptyset$
(i.e., $Y\subset X$).
If $F\cap X\ne \emptyset$,
then
$Y$, together one vertex in $F\cap X$,  forms one shore of a 2-separation cut that is laminar with $C$, contradicting the assumption of $C$.
Thus, $F\cap X=\emptyset$,
implying that
$F\subset \overline{X}$.
By Corollary \ref{cor:connected}, $G[X]$ is connected. So $X=Y$. However, $|Y|$ is even, while $|X|$ is odd.
%Then,	the set $\overline{X}\setminus \{x_0\}$ for some vertex vertex $x_0$ in $F\cap \overline{X}$}
%forms one shore of a 2-separation cut that is laminar with $C$.
This is a contradiction too.
Therefore, $Y\cap \overline{X}\neq \emptyset$.

 Similarly,
 $Y\cap X\neq \emptyset$.
 Noting that both $G[\overline{X}]$ and $G[{X}]$ are connected by Corollary \ref{cor:connected}, we have
 $|F\cap \overline{X}|=1$.

Let $t\in F\cap \overline{X}$ and $Y'=Y\cup F$.
Then $\partial(\overline{Y'}\cup\{t\})$ is a 2-separation cut associated with $F$. Adjust notation such that $|X\cap (Y'\setminus \{t\})|$ is odd. So $\partial (X\cap (Y'\setminus \{t\}))$ is tight in $G$ by Proposition \ref{pro:ctight}.
Let $G'=G/((\overline{Y'}\cup\{t\})\rightarrow \overline{t})$. Then $\partial(X\cap V(G'))$ is also a  tight cut of $G'$ by Proposition \ref{pro:also_tight} (note that $X\cap V(G')=X\cap (Y'\setminus \{t\})$). As $Y\cap X\neq \emptyset$, $\partial(X\cap V(G'))$ is non-trivial.
By Theorem \ref{thm:ELP-th}, $G'$ contains a non-trivial ELP-cut $C'$.

If $C'$ is a barrier-cut associated with a non-trivial barrier $B'$ of $G'$, let $B''=B'$ if
$\overline{t}\notin B'$; otherwise, let $B''=(B'\cup \{t\})\setminus\{ \overline{t} \}$.
Then $B''$ is also a non-trivial barrier of $G$ by Proposition \ref{lem:2-sep}.
 This is a contradiction to the assumption that $G$ is bicritical,
 as a bicritical graph contains no
 non-trivial barriers (see Theorem 5.2.5 in \cite{Lovasz86} for example).
 It follows that $C'$ is not a barrier-cut, implying $C'$ is a 2-separation cut
  of $G'$. Then $C'$ is also a 2-separation cut of $G$ by Proposition \ref{lem:2-sep},
 contradicting the fact  that $F\cup Y$ contains no 2-separation of $G$ other than $F$ (obviously, the 2-separation that $C'$ is associated with is different from $F$).

Now we consider the case that $G$ is not bicritical.
Let $B$ be a maximal non-trivial barrier of $G$,
and let $G_1, G_2,\cdots, G_{a+b}$ be the components of $G-B$,
where $|X\cap V(G_i)|$ is odd for
all $i=1,2,\dots,a$
and $|{X}\cap V(G_i)|$ is even for
all $i=a+1,a+2,\dots,a+b$,
as shown in Figure~\ref{f3}.
Then $a+b=|B|$. Adjust notation so that $a\neq 0$.

\begin{figure}[h]
	\centering
	\includegraphics[width=12 cm] {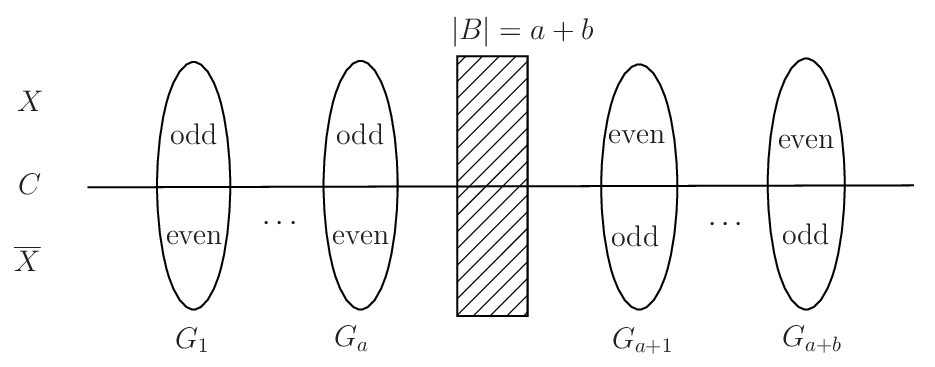}
	\caption{$B$ is a non-trivial barrier and $G_1, G_2,\cdots, G_{a+b}$ are the components of $G-B$}
	\label{f3}
\end{figure}

\noindent {\bf Claim 1.1}:
$\partial({X} \cap \left (\cup_{i=1}^{a} V(G_i) \right)) \cap C \ne  \emptyset$, $|B \cap X| = a -1$
and $|B \cap \overline{X}| = b+1$.

Suppose, to the contrary, that $\partial({X} \cap \left (\cup_{i=1}^{a} V(G_i) \right)) \cap C = \emptyset$.
%If $\partial({X} \cap \left (\cup_{i=1}^{a} \red{V(G_i)} \right)) \cap C = \emptyset$,
Then there exists an edge
$st\in E_G(X, \overline{X}\setminus B)$,
where $s\in X$ and
$t\in \overline{X}\setminus B$,
otherwise $G$ is not connected as
$B$ is an independent set by Proposition \ref{pro:B_in_mc}.
 %$E(G[B])=\emptyset$.
Let $M$ be a perfect matching of $G$ containing $st$.
Note that $s\in B$ or $s$ lies in some component of $G-B$, say $G_j$, where $j\in \{a+1,a+2,\dots,a+b\}$. As $B$ is a barrier of $G$,	$M$ contains exactly one edge in $E[V(G_i),B]$ for
	each $i=1,2,\dots, a+b$. If $s\in V(G_j)$, then $|E[V(G_j)\cap X, B\cap X]\cap M|=1$ as  $|V(G_j)\cap X|$ is even, $C\cap M=\{st\}$ and $C$ is tight;  if $s\in B$, then $E[V(G_j)\cap X, B\cap X]\cap M=\{st\}$.
Note that $\partial({X} \cap \left (\cup_{i=1}^{a} V(G_i) \right)) \cap C = \emptyset$
implies that
$(\cup_{i=1}^{a} V(G_i))\cap \overline{X}=\emptyset$.
As $G_i$ is an odd component of $G-B$ for $i=1,2,\dots,a+b$, and $B$ is a barrier of $G$,
we have $|E[V(G_i)\cap X, B\cap X]\cap M|=1$ for $i=1,2,\dots,a$, and $|E[V(G_i)\cap \overline{X},B\cap \overline{X}]\cap M|=1$ for each $i\in \{a+1,a+2,\dots, a+b\} $ and $i\neq j$.
So, $|B\cap X|=a+1$. Therefore, $B\cap X$ is a $C$-sheltered barrier of $G$.
By the assumption of $C$, $B\cap X$ is a trivial barrier, that is $a=0$. This is a contradiction to the assumption of $a$.

By the conclusion in the previous paragraph, we may assume that
%$\partial({X} \cap (\cup_{i=1}^{a} V(G_i))) \cap C \neq \emptyset$
$e_1 \in \partial(X \cap (\cup_{i=1}^{a} V(G_i))) \cap C$.	
	Let $M_1$ be a perfect matching of $G$ containing $e_1$.
	Adjust notation so that one end point of $e_1$ lies in $V(G_1)\cap X$. Then the other end point of $e_1$ lies in $V(G_1)\cap \overline{X}$ or $B\cap \overline{X}$.
	As $B$ is a barrier of $G$,
	$M_1$ contains exactly one edge in $E[V(G_i),B]$ for
	each $i=1,2,\dots, a+b$.
	As $e_1 \in M_1\cap C$ and $C$ is tight, $E[V(G_i),B]\cap M_1\cap C=\emptyset$ for $i=2,3,\dots, a+b$. Recalling that  $|V(G_i)\cap X| $  is odd  for all $i=2,3,\dots, a$ and $| V(G_i)\cap\overline{X}|$ is odd for
all $i=a+1,a+2,\dots,a+b$, we have
$|E[V(G_i)\cap X,B\cap X]\cap M_1|=1$ for $i=2,3,\dots, a$
	and  $|E[V(G_i)\cap \overline{X},B\cap \overline{X}]\cap M_1|=1$ for $i=a+1,a+2,\dots, a+b$.
 So $|B \cap X| = a -1$ or $a$. As $|V(G_1)\cap \overline{X}| $  is even,  $|E[V(G_1),B\cap \overline{X}]\cap M_1|=1$.
 Therefore, $|B \cap X| = a -1$, and thus $|B \cap \overline{X}| = b+1$. 	
Hence Claim 1.1 holds.

%\red{We are now going to show that $b=0$.}Suppose that $b>0$,and

By replacing $X$ by $\overline{X}$
in the first paragraph of the proof of Claim 1.1,
%we have the conclusion that
$\partial(\overline{X} \cap (\bigcup_{i = a+1}^{a+b} V(G_i)))\cap C\neq \emptyset$ or $b=0$.
Now we are going to the following claim.

\noindent {\bf Claim 1.2}:
$b=0$.

Suppose that $b\neq 0$.
Since $\partial(\overline{X} \cap (\bigcup_{i = a+1}^{a+b} V(G_i)))\cap C\neq \emptyset$,
assume, without loss of generality, that $e_2 \in \partial(\overline{X} \cap V(G_{a+1}))\cap C$.	
	Let $M_2$ be a perfect matching containing $e_2$. As $C$ is a tight cut, $e_2\in C$ and $|V(G_{a+1})\cap X|$ is even,
$|E[V(G_{a+1}), B\cap X]\cap M_2|=1$. For all $i=1,2,\ldots,a$, as  $|V(G_{i})\cap X|$ is odd, $B$ is a barrier and
$M_2\cap C=\{e_2\}$,
 we have $|E[V(G_i), B\cap {X}]\cap M_2|=1$. Then $|{X} \cap B| \geq a+1$. Therefore,  $|\overline{X} \cap B| \leq  b-1$.
	This is a contradiction to Claim 1.1.
	Hence $b=0$ and Claim 1.2 holds.

\iffalse
	for all $i=a+2,a+3,\ldots,a+b$. So $|E[V(G_i), B\cap \overline{X}]\cap M_2|=1$
	for all $i=a+2,a+3,\ldots,a+b$. Similarly, $|E[V(G_i), B\cap {X}]\cap M_2|=1$
	for all $i=1,2,\ldots,a$.

	If $\partial(\overline{X} \cap (\bigcup_{i = a+1}^{a+b} V(G_i)))\cap C = \emptyset$, then each $G_i(i=a+1,a+2,\dots,a+b)$ is a component of
  $G - (\overline{X} \cap B)$.
    Therefore, $\overline{X} \cap B$ is a barrier of $G$ that is $C$-sheltered. By assumption, $\overline{X} \cap B$ is a trivial barrier,  that is $ |B\cap \overline{X}|=1$. So $b = 0$.
    \fi

\noindent{\bf Claim 1.3}: There exists
$i$ with $1\le i\le a$ such that
$\partial(X \cap V(G_i))$ is  a
non-trivial tight cut of $G/(\overline{V(G_i)})$.

By Claims 1.1 and 1.2, $|B\cap \overline{X}|=1$.
Assume that
$y$ is the only member in
$\overline{X}\cap B $.
	If $V(G_i) \cap \overline{X} = \emptyset$
	for every $i = 1,\cdots, a$,
	then $b=0$ yields that $\overline{X}=\overline{X}\cap B$, implying that
	$C$ is a trivial tight cut, contradicting the assumption.
	So there exists $i$ with $1\le i\le a$
	such that $V(G_i) \cap \overline{X} \neq \emptyset$,
		say $i=1$.  As $\partial(V(G_1))$ and $\partial(X)$ are  tight cuts of $G$, and $|X \cap V(G_1)|$ is odd, $\partial(X \cap V(G_1))$ is a tight cut of $G$ and $E[V(G_1)\cap \overline{X}, \overline{V(G_1)}\cap {X}]=\emptyset$ by Proposition \ref{pro:ctight}. Let $G'=G/(\overline{V(G_1)})$.
	Then
$\partial(X \cap V(G_1))$ is a tight cut of $G'$ by Proposition \ref{pro:also_tight}.	
	If $|X \cap V(G_1)| = 1$, assume that $x \in X \cap V(G_1)$, then $N_G(V(G_1)\cap \overline{X})=\{x,y\}$, that is $\{x,y\}$ is a $2$-separation of $G$ (note that $E[V(G_1)\cap \overline{X}, B\cap {X}]=E[V(G_1)\cap \overline{X}, \overline{V(G_1)}\cap {X}]=\emptyset$). Therefore, $\partial ((V(G_1)\cap \overline{X})\cup \{y\})$ is a  $2$-separation cut associated with $\{x,y\}$ which is laminar with $C$,  contradicting the assumption of $C$.
	Thus, $|X \cap V(G_1)| > 1$,
	implying $\partial(X \cap V(G_1))$ is  a non-trivial tight cut of $G'$.
	Claim 1.3 holds.
	
	By applying Claim 1.3,  adjust notation such
	that $\partial(X \cap V(G_1))$ is  a
	non-trivial tight cut of $G':=G/(\overline{V(G_1)}\rightarrow g)$. Then $g\in N_{G'}(X)$.
	
	\noindent {\bf Claim 1.4}:
		$\partial(X \cap V(G_1))$ is
		not a barrier-cut of $G'$. %associated with any barrier $B'$.

	Suppose that
	$\partial(X \cap V(G_1))$ is a barrier-cut of $G'$ associated with the non-trivial barrier $B'$ of $G'$.
If $g\notin B'$, then $g$ lies in some component of $G-B'$, say $G_1'$. As $G[X]$ is connected (by Corollary \ref{cor:connected}), $g\in N_{G'}(X)$.
As
$\partial(X)$ is a barrier-cut of $G'$, $V(G_1')=V(G')\setminus {X}$ and $B'\subset X$.
Then $B'$ is also a $C$-sheltered barrier of  $G$ such that $G[(V(G'_1)\setminus \{g\})\cup (V(G)\setminus V(G_1))]$ is a component of $G-B'$, contracting the assumption of $C$. So $g\in B'$.
 It can be checked that   $(B'\setminus\{g\})\cup B $ is a barrier of $G$,
	contradicting the fact that $B$ is maximal  barrier of $G$.
	Hence Claim 1.4 holds.
	
	Now we are going to complete the proof of Claim 1.
	
Since $G'$ has less tight cuts than $G$,
by induction and Claim 1.4, $\partial(X \cap V(G_1))$ is an essential $GS$-cut of $G'$.
	If $\partial(X \cap V(G_1))$ is not a $GS$-cut of $G'$, then there exists some non-trivial barrier of $G'$ that is $\partial(X \cap V(G_1))$-sheltered. And it is also a $C$-sheltered barrier of $G$, contradicting the assumption. So $\partial(X \cap V(G_1))$ is  a $GS$-cut of $G'$. Assume that $ \mathcal{F}'$ is
the set of  2-separations  of $G'$ associated with $\partial(X \cap V(G_1))$.
	Then, by the proof of Theorem~\ref{thm:lamnum},
		there exists two $2$-separations $F'$ and $F''$ in $ \mathcal{F}'$ such that  both $\partial (Y')$ and $\partial (Y'')$ are laminar with $\partial(X \cap V(G_1))$,
		where  $\partial (Y')$ and $\partial (Y'')$  are $2$-separation cuts associated with $F'$ and $F''$ respectively,
	such that
		$(Y'\setminus F')\cap (Y''\setminus F'')=\emptyset$
		and
		$Y'\setminus F'$ (resp. $Y''\setminus F''$) is an even component of $G'-F'$ (resp. $G'-F''$).

If $g\notin F'$, then both vertices in $F'$  are also vertices of $G$,
	implying that $g\in Y'\setminus F'$, otherwise $\partial (Y')$ is a $2$-separation cut of $G$ laminar with $C$ in $G$, contradicting  the assumption of $C$.  Then $F''\cup Y''\subset V(G)$, and thus
$\partial (Y'')$ is a $2$-separation cut of $G$ laminar with $C$ in $G$, contradicting  the assumption of $C$ again.
So, $g\in F'$. With the same reason,  we have $g\in F''$.
Therefore, every $2$-separation $F_1$ in $ \mathcal{F}'$,  satisfying  one of $2$-separation cut associated with $F_1$ is laminar with $\partial(X \cap V(G'))$,
 contains the vertex $g$.
 Note that every end-2-separation of $ \mathcal{F}'$  results in a $2$-separation cut laminar with $\partial(X \cap V(G'))$
    by the proof of Theorem~\ref{thm:lamnum}.
 By the structure of the $GS$-cut, it follows that either
$\overline{X}\cap V(G')=\{g\}$
or $V(H)\subseteq \overline{X}$
for some even component $H$ of $G'-F_2$, where $F_2\in \mathcal{F}'$. % lie in $\overline{X}$.
However, if $\overline{X}\cap V(G')=\{g\}$, then
$|\overline{X}\cap V(G)|=1$
and $C$ is a trivial cut in $G$,
a contradiction.
Then, $V(H)\subseteq \overline{X}$,
implying that
$\partial(V(H)\cup \{g\})$ is a 2-separation cut of $G'$.

Recalling that both $\partial (X)$ and $\partial (V(G_1))$ are tight in $G$ and $|X\cap V(G_1)|$ is odd,
and
by Proposition \ref{pro:ctight},
$E[X\cap \overline{V(G_1)}, \overline{X}\cap V(G_1)]=\emptyset$.
%no edge connects $X\cap \overline{V(G_1)}$ to  $\overline{X}\cap V(G_1)$.
As $V(H)\subset \overline{X}\cap V(G_1)$
and $X\cap B\subset X\cap \overline{V(G_1)}$,
$E[X\cap B, V(H)]=\emptyset$. So
 $N_G(V(H))\cap X=F_2\setminus\{g\}$. Therefore,
 $(F_2\setminus\{g\})\cup\{ y\}$ is a 2-separation of $G$ such that $H$ is a component of $G- (F_2\cup\{ y\})\setminus\{g\}$.

Hence
 $\partial(V(H)\cup \{y\})$ is a 2-separation cut of $G$ that laminar with $C$. This is a contradiction
to the assumption given in the beginning of the proof in Claim 1.
\end{proof}

%Finally, we apply Claim 1 to complete the proof of the theorem.
By Claim 1,
$G$ has a $C$-sheltered non-trivial barrier or a 2-separation cut that is laminar with $C$.
We will prove the theorem holds
in both cases.

\noindent {\bf Case 1.} There exists a $C$-sheltered non-trivial barrier $B$.

Adjust notation so that $G_1$ is  a non-trivial component of  $G-B$ with $X\subset V(G_1)$. Let $G'=G/(\overline{V(G_1)}\rightarrow t)$. Then
$C$ is also a tight cut of $G'$ by Proposition \ref{pro:also_tight}. As $C$ is not a barrier-cut of $G$, the edge cut $C$ is non-trivial in $G'$.
If $C$ is a barrier-cut of $G'$, let $B'$ be the barrier of $G'$
associated with $C$. If $t\in B'$, then $G'[X]$ is a component of $G'-B'$ as $C$ is a barrier-cut of $G'$ and $t\notin X$. It can be checked that $(B\cup B')\setminus \{ t\}$ is also a barrier of $G$ such that $G[X]$ is a component of $G-(B\cup B')\setminus \{ t\}$. And thus  $C$ is a barrier-cut of $G$, contradicting the assumption on $C$. Therefore,   $t\notin B'$.
 Then $t$ lies in some component of $G'-B'$, say $G_1'$.
Recalling that $C$ is  a tight cut of $G'$, $G[X]$ is a component of $G'-B'$, or $V(G_1')=V(G')\setminus X$ (in this case, $B'\subset X$). Then  $B'$ is also a barrier of $G$
 with the property that
 	$G[X]$ is a component of $G'-B'$, or $G[V(G)\setminus X]$ is a component of $G-B'$. So $\partial(X)$ is also a barrier-cut of $G$,
  contradicting the assumption on $C$ again.
Hence $C$ is an essential $GS$-cut of $G'$.
It can be further verified that
no matter whether
$t$ lies  in a 2-separation $F$ of $G'$ satisfying $|F\cap X|=1$,
$C$ is an essential $GS$-cut of $G$ by its definition.

\noindent {\bf Case 2.} There exists a 2-separation cut $D$ of $G$  that is laminar with $C$.

Adjust notation so that $X\subset Y$, where $Y$ is one shore of $D$.
%where $D=\partial (Y)$.
Let
 $\{t_1,t_2\}$ be the 2-separation of $G$ that $D$ associated with, and let  $t_1\in Y$.
Let $G'=G/(\overline{Y}\rightarrow t)$. Then $t_1\in V(G')$ and $tt_1\in E(G')$.  Let $C'=\partial(X\cap V(G'))$. Then  $C'$ is a tight cut of $G'$ by Proposition \ref{pro:also_tight}.

\noindent {\bf Case 2.1.}
$C'$ is a barrier-cut of $G'$.

In this subcase,  let $B'$ be the barrier of $G'$
associated with $C'$. As $tt_1\in E(G')$,
at most one vertex in $\{t,t_1\}$
lies in $B'$ by Proposition \ref{pro:B_in_mc}.
As $C$ is not a barrier-cut of $G$,
either $t\in B'$ or $t_1\in B'$ (otherwise both $t$ and $t_1$ lie in some component of $G'-B'$, and
$B'$ is a barrier of $G$ such that $G[X]$ or $G[V(G)\setminus X]$ is a component of $G-B'$,
which implies that $C$ is  a barrier-cut of $G$).

If $t\in B'$, then  $G'[X]$ is a component of $G'-B'$ as $C'$ is a barrier-cut and $t\notin X$. And
 $(B'\cup \{t_2\})\setminus\{t\}$ is also a barrier of $G$. If $t_1\notin X$, then $G[X]$ is a component of $G-((B'\cup \{t_2\})\setminus\{t\})$, implying that $C$ is also a barrier-cut of $G$, contradicting the assumption of $C$. So, $t_1\in X$.
 Then $G[(X\cup \overline{Y})\setminus\{t_2\}] $ is a component of  $G-((B'\cup \{t_2\})\setminus\{t\})$.
Let $G''=G/(((\{t_2\}\cup Y)\setminus  X)\rightarrow s)$. Note that $G''$ also can be gotten from $G$ by contracting  the barrier $(B'\cup \{t_2\})\setminus\{t\}$ into the vertex $s$.
  Then $\{s,t_1\}$ is a 2-separation
  cut of $G''$ which partitions
  $G''- \{s,t_1\}$
  into two graphs with vertex sets $X\setminus \{t_1\}$ and $\overline{Y}\setminus \{t_2\}$,
 respectively.
  Therefore,  $C$ is an essential $GS$-cut of $G$, as $(B'\cup \{t_2\})\setminus\{t\}$ is a $C$-sheltered  barrier of $G$.

  %\noindent {\bf Case 2.2}: $t_1\in B'$. $\partial(Y\setminus (t_2\cup X))$ is a
If $t_1\in B'$, then $t$ lies in some component of $G'-B'$, say $G_1'$.
 %As $X\subset Y$, the subgraph $G'[V(G')\setminus X]$ is a component of $G'-B'$ and $t\in V(G')\setminus X$.
 So, $B'$ is also a barrier of $G$ such that $G[(V(G_1')\cup \overline{Y})\setminus \{t\}]$ is an odd component of  $G-B'$.
 Note that $C'$ is a barrier-cut of $G'$.
 If $G'[X]$ is an odd component of  $G'-B'$, then $G[X]$ is also an odd component of  $G-B'$ (note that $t\notin X$); if $G'[V(G')\setminus{X}]$ is an odd component of  $G'-B'$ (in this case, $t_1\in X$), then $G[V(G)\setminus{X}]$ is also an odd component of  $G-B'$ (in this case, $V(G)\setminus{X}=(V(G')\cup \overline{Y})\setminus ({X}\cup\{t\})$).
 Therefore, $C$ is a barrier-cut of $G$, contradicting the assumption on $C$.
  Hence the theorem holds
  in Case 2.1.

  \noindent {\bf Case 2.2.}
  $C'$ is not a barrier-cut of $G'$.

In this subcase,  $C'$ is an essential $GS$-cut  of $G'$ by induction. As $X\subset Y$, $t\notin X$ (and $t_2\notin X$).
 If  $\{t_1,t\}\subset \overline{X}$,
 then $C'$ is also an essential $GS$-cut  of $G$. So we assume that $t_1\in X$.  As $tt_1\in E(G')$, $tt_1\in C'$. We firstly assume that  $C'$ is   a  $GS$-cut  of $G'$.
Let $\mathcal{F}$ be
the set of  2-separations  of $G'$
associated with $C'$. Note that for every edge $e$ in $C'$, at least one end point of $e$ lie in some 2-separation of $\mathcal{F}$. As $tt_1\in C'$, at least one of $t$ and $t_1$ lie in some 2-separation of $\mathcal{F}$. In the following, we may assume
that $t_1\in F_1$ for some
$F_1\in \mathcal{F}$.

 \noindent {\bf Case 2.2.1.}
 $\{t_1,t\}\in  \mathcal{F}$,
 or
$t_1\in F_1$ for some $F_1\in \mathcal{F}$ but
$t\notin F$ for any $F\in
\mathcal{F}$.

In this subcase, as $X\subset Y$, $\overline{Y}\subset \overline{X}$. Then the edge cut
$C$ is a $GS$-cut  of $G$ and    $\mathcal{F}\cup \{t_1,t_2\}$ is the set of 2-separations associated with $C$.

\iffalse
 If exactly one of  $t_1$ and $t$  lie in some 2-separation in $ \mathcal{F}$, or $D$ is also a 2-separation in $ \mathcal{F}$, then $C$ is an essential $GS$-cut  of $G$ associated with $\mathcal{F}\cup D$, as $D$ is laminar with $C$.
 \fi

\noindent {\bf Case 2.2.2.}
$t_1\in F_1$ for some $F_1\in \mathcal{F}$ but
$t\in F_2$ for some $F_2\in
\mathcal{F}$,
where $F_1\ne F_2$.

 %Now we consider the case when $t_1\in F_1$ and $t_2\in F_2$ where $ \{F_1, F_2\}\subset \mathcal{F}$.

 In this subcase,
 there exists a set of  2-separations $\mathcal{F'}=\{F'_1,F'_2,\dots, F'_s\}\subseteq \mathcal{F}$
such that  $F_1=F'_1$, $F_2=F'_s$  and $|F_{i}\cap F_{i+1}|=1$ for $i=1,2,\dots,s-1$.
Since $t_1$ and $t$ lie in different shores of $C$, we have $s\geq 3$. Assume that $F_3\in \mathcal{F'}\setminus\{F_1,F_2\}$. Then $t_1$ and $t$ lie in different components of $G'-F_3$. However, $t_1t\in E(G')$. This is a contradiction. So the result holds
in this subcase.

Now we assume that $C'$ is   an essential  $GS$-cut   that is not a $GS$-cut of $G'$. Assume  that ${\cal B}=\{B_1,B_2,\ldots,B_s\}$ is the set of  barriers of $G'$ associated with $C'$,  $B_i\subset X$ for all $i=1,2\ldots,t$ and $B_i\subset \overline{X}$ for  all $i=t+1,t+2\ldots,s$. For $i=1,2\ldots,s$, we contract each barrier $B_i$  into a vertex $b_i$ respectively, denoted the resulting graph by $G''$. Let $X''=(X \setminus \cup_{i=1}^t B_i)
\cup \{b_i: 1\le i\le t\}$.
If $t_1$ lies in some barrier in ${\cal B}$, say $B_1$, then let $t_1''=b_1$, otherwise, let $t_1''=t_1$. Similarly, if $t$ lies in some barrier in ${\cal B}$, say $B_2$, then let $t''=b_2$, otherwise, let $t''=t$.
 By replacing  $\{t,t_1\}$ by $\{t'',t_1''\}$ in the Case 2.2 when $C'$ is   a $GS$-cut of $G'$, it can be shown that $\partial (X'')$ is a $GS$-cut of $H$, where $H$ is the graph gotten from $G$ by contracting each barrier in ${\cal B}$. By the definition of an essential $GS$-cut, it can be checked that $C$ is an essential $GS$-cut of $G$.
Hence we complete the proof.
%%%%%%%%%%%%%%%%%%%%%%%%%%%%%%%%%%%%%%%%%%%%%%%%%%%%%%%%%%%%%%%%%%%%%%%%%%%%%%
\end{proof}

\section{Remarks}

The proof of Theorem \ref{main} does not rely on Theorem \ref{thm:lam}. By the structure of an essential $GS$-cut,  Theorem \ref{thm:lam} can be inferred by Theorem \ref{main} (note that if $C$ is a barrier-cut, then $C$ is an ELP-cut  laminar with itself). Recall that if $C$ is a non-ELP tight cut of a matching covered graph $G$, then,  by Theorem \ref{thm:lam}, $|\ELP_G(C)|\geq 1$. And the lower bound is sharp by
 the remarks of Theorem \ref{thm:lamnum}. If $C$ is a barrier-cut of $G$, $|\ELP_G(C)|$ may also be 1. For example, in $H_n'$, $C'=\partial(\{v_1,v_2,\ldots, v_{2n+1}\})$ is the only ELP-cut.  So $|\ELP_{H_n'}(C')|=1$.

%2. By Theorem \ref{main} and the structure of an essential $GS$-cut, if $C$ is a non-trivial tight cut of a matching covered graph $G$ and $e\in C$, then  at least one end point of $e$ lie in a non-trivial barrier or 2-separation of $G$.

%%%%%%%%%%%%%%%%%%%%%%%%%%%%%%%%
%%%%%%%%%%%%%%%%%%%%%%%%%%%%%%

%%%%%%%%%%%%%%%%%%%%%%%%%%%%%%%%%%%%%%%%%##############################

%\section{Concluding remarks}
%\label{sec:cr}

%\begin{figure} [htbp]
%\begin{center}
%\includegraphics[width=8cm]{countex.eps}
%\end{center}
%\caption{  A 3-connected cubic matching covered graph without removable edges }\label{fig:ce}
%\end{figure}

\end{document}